\documentclass[11pt,a4paper,reqno]{amsart}
\usepackage[latin1]{inputenc}

\usepackage{amsmath,amsthm,amssymb,amsfonts}

\usepackage{esint}
\usepackage{amsmath}
\usepackage{amsfonts}
\usepackage{amssymb}
\usepackage[vmargin=3.5cm]{geometry}

\usepackage{hyperref}
\usepackage{todonotes}
\allowdisplaybreaks

\usepackage[abbrev,backrefs]{amsrefs}
\usepackage[capitalize]{cleveref}

\numberwithin{equation}{section}

\newtheorem{proposition}{Proposition}[section]
\newtheorem{theorem}[proposition]{Theorem}
\newtheorem{lemma}[proposition]{Lemma}

\theoremstyle{definition}
\newtheorem{definition}[proposition]{Definition}

\theoremstyle{remark}
\newtheorem{remark}[proposition]{Remark}
\newcommand{\defeq}{\triangleq}

\newcommand{\Nset}{\mathbb{N}}

\newcommand{\Rset}{\mathbb{R}}
\newcommand{\dif}{\,\mathrm{d}}
\newcommand{\abs}[1]{\lvert #1 \rvert}

\newcommand{\biggabs}[1]{\biggl\lvert #1 \biggr\rvert}
\newcommand{\norm}[1]{\lVert #1 \rVert}
\newcommand{\compose}{\,\circ\,}

\DeclareMathOperator{\supp}{supp}

\newcommand{\st}{\ ;\ }

\usepackage{constants}

\author{Rémy Rodiac and Jean Van Schaftingen}

\title{
  Metric characterization of the sum of fractional Sobolev spaces}

\date{}

\address{
Universit\'e catholique de Louvain, Institut de Recherche en Math\'ematique et Physique, Chemin du Cyclotron 2 bte L7.01.01, 134 8 Louvain-la-Neuve, Belgium}
\email{remy.rodiac@uclouvain.be, jean.vanschaftingen@uclouvain.be}



\begin{document}

\begin{abstract}
We introduce a non-linear criterion which allows us to determine when a function can be written as a sum of functions belonging to homogeneous fractional spaces: for $\ell \in \mathbb{N}^*$, $s_i\in (0, 1)$ and $p_i \in [1, +\infty)$, $u : \Omega \to \mathbb{R}$ can be decomposed as 
$u = u_1+\dotsc+u_\ell$ with \ $u_i  \in \dot{W}^{s_i,p_i}(\Omega)$ if and only if 
$$
\iint\limits_{\Omega \times \Omega} \min_{1 \le i \le \ell} \frac{|u (x) - u (y)|^{p_i}}{|x - y|^{n+s_ip_i}}\,\mathrm{d}x \,\mathrm{d}y
<+\infty.
$$
\end{abstract}

\thanks{Both authors were supported by the Mandat d'Impulsion Scientifique F.4523.17, ``Topological singularities of Sobolev maps'' of the Fonds de la Recherche Scientifique--FNRS}

\maketitle

\section{Introduction}

Given $\Omega$  an open set of $\Rset^n$ or an $n$--dimensional Riemannian manifold with \(n\geq 1\), $s \in (0, 1)$ and $p \in [1, +\infty)$,
the \emph{homogeneous fractional Sobolev space} \(\dot{W}^{s, p} (\Omega)\) (or Slobodeskii space) is defined as 
\begin{equation*}
  \dot W^{s,p}(\Omega)
  \defeq
  \bigl\{ u : \Omega \to \Rset \st \text{\(u\) is measurable and } |u|_{W^{s,p} (\Omega)} <+\infty \bigr\},
\end{equation*}
where the \emph{Gagliardo semi-norm} of a measurable function \(u : \Omega \to \Rset\) is defined as
\begin{equation}
  \label{eq_Gagliardo_Norm}
  |u|_{W^{s,p} (\Omega)}^p\defeq \iint\limits_{\Omega \times \Omega} \frac{|u(x)-u(y)|^p}{d_\Omega (x, y)^{n+sp}} \dif x \dif y,
\end{equation}
and \(d_\Omega (x, y)\) denotes the Euclidean distance between the points \(x \in \Omega\) and \(y \in \Omega\)
when \(\Omega \subset \Rset^n\) or their geodesic distance when \(\Omega\) is a Riemannian manifold.

Fractional Sobolev spaces are linear spaces which can be summed, given \(\ell \in \mathbb{N}^*\), \(s_1, \dotsc, s_\ell \in (0, 1)\)
and \(p_1, \dots, p_\ell \in [1, +\infty)\), as 
\begin{equation}
  \label{eq_sum_spaces}
\dot{W}^{s_1, p_1} (\Omega) + \dotsb + 
\dot{W}^{s_\ell, p_\ell} (\Omega)
\defeq \bigl\{u_1 + \dotsb + u_\ell \st \text{for each \(i \in \{1, \dotsc, \ell\}\), } u_i \in \dot{W}^{s_i, p_i} (\Omega)\bigr\}.
\end{equation}
While the definition of the Gagliardo semi-norm \eqref{eq_Gagliardo_Norm} extends readily to the case where the target space \(\Rset\) is replaced by any metric space, the definition of the sum \eqref{eq_sum_spaces} relies strongly on the linear structure of the space.
The goal of the present work is to give the sum of fractional Sobolev spaces a metric characterization which does not depend on the linear structure of the space \(\Rset\) and thus to pave the way to a definition of the sum of some nonlinear spaces.

\begin{theorem}\label{theorem_main}
If either $\Omega=\Rset^n$, or $\Omega \subset \Rset^n$ is a Lipschitz bounded open set, or $\Omega$ is an $n$--dimensional Lipschitz compact manifold with a possibly empty boundary, then 
\begin{multline*}
  \dot{W}^{s_1, p_1} (\Omega) + \dotsb + \dot{W}^{s_\ell, p_\ell} (\Omega) \\
  = 
  \Biggl\{ u : \Omega \to \Rset \st \text{\(u\) is measurable and }  \iint\limits_{\Omega \times \Omega}
  \min_{1 \le i \le \ell} \frac{\abs{u (x)-u (y)}^{p_i}}{d_\Omega (x, y)^{n + s_i p_i}}
  \dif x \dif y
  <+\infty
  \Biggr\}.
\end{multline*}
\end{theorem}

\Cref{theorem_main} is closely linked to its counterpart for the \emph{intersection} of fractional Sobolev spaces
\begin{multline}
\dot{W}^{s_1, p_1} (\Omega) \cap \dotsb \cap  \dot{W}^{s_\ell, p_\ell} (\Omega) \\
= 
\biggl\{ u : \Omega \to \Rset \st \text{\(u\) is measurable and }    \iint\limits_{\Omega \times \Omega}
\max_{1 \le i \le \ell} \frac{\abs{u (x)-u (y)}^{p_i}}{d_\Omega (x, y) ^{n + s_i p_i}}
\dif x \dif y <+\infty
\biggr\},
\end{multline}
whose proof is direct.
The counterpart of \cref{theorem_main} for \(L^p\) spaces is 
\begin{equation}
  \label{lp_decomposition}
L^{p_1} (\Omega) + \dotsb + L^{p_\ell} (\Omega) = 
\biggl\{ u : \Omega \to \Rset \st \text{\(u\) is measurable and }    \int_{\Omega}
\min_{1 \le i \le \ell} \abs{u}^{p_i} <+\infty
\biggr\}.
\end{equation}

Our study of the problem was motivated by the appearance of the space \( W^{s, p} (\Omega) + W^{1, sp} (\Omega)\)
in the lifting of maps in \(W^{s, p} (\Omega, \mathbb{S}^1)\) \citelist{\cite{Bourgain_Brezis_Mironescu2000}\cite{Mironescu_preprint}\cite{Mironescu2008}\cite{Bourgain_Brezis2003}\cite{Nguyen2008}}.
A first step of the generalization of these results to liftings over a general covering of a manifold \cite{Bethuel_Chiron2007} is the definition of an appropriate nonlinear counterpart of \(W^{s, p} (\Omega) + W^{1, sp} (\Omega)\). The present work shows how to define this for fractional Sobolev spaces.

\bigbreak 

The proof of the inclusion of the sum in \cref{theorem_main} is quite straightforward and independent on any assumption on the domain \(\Omega\) (\cref{proposition_estimate_sum}). The associated estimate has a constant that remains bounded if the number \(\ell \in \Nset\) and the exponents \(p_1, \dotsc, p_\ell \in [1, +\infty) \) remain bounded.

For the converse inclusion, we first recall that the function \(u : \Omega \to \Rset\) in the right-hand side of \eqref{lp_decomposition} can be decomposed by defining the functions \(u_1, \dotsc, u_\ell\) in such a way that for every \(x\) and every \(j \in \Nset\), one has \(u_j (x) \in \{0, u (x)\}\) and \(\sum_{j = 1}^\ell \abs{u_j (x)}^{p_j} = \min_{1 \le i \le \ell} \abs{u_i (x)}^{p_i}\).
Such an approach fails for fractional Sobolev spaces because these spaces are defined by a double integral, and this strategy would provide a decomposition on \(\Omega \times \Omega\) rather than on \(\Omega\).

We tackle the problem through the characterization of fractional Sobolev spaces as traces of weighted Sobolev spaces \citelist{\cite{Gagliardo_1957}\cite{Uspenskii}\cite{Mironescu_Russ2015}}. 
When \(\Omega = \Rset^n\), we show that any measurable function \(u : \Rset^n \to \Rset\), has an extension \(U : \Rset^{n + 1}_+ \to \Rset\) such that 
\[
\iint\limits_{\Rset^{n + 1}_+}
\min_{1 \le i \le \ell} \frac{\abs{\nabla U (x, t)}^{p_i}}{t^{1 - (1 - s) p_i}} \dif t \dif x
\le 
C
\iint\limits_{\Rset^n \times \Rset^n}
\min_{1 \le i \le \ell} \frac{\abs{u (x) - u (y)}^{p_i}}{\abs{x - y}^{n + s_i p_i}} \dif x \dif y.
\]
We then decompose the derivative \(\nabla U\) of this extension into functions \(\Theta_1, \dotsc, \Theta_\ell : \Rset^{n + 1}_+ \to \Rset^{n + 1}\) that satisfy weighted estimates; these are not necessarily derivatives but can be used to construct functions \(u_1, \dotsc, u_\ell : \Rset^n \to \Rset\) that are controlled by some trace estimates.
The resulting estimates blow up when \(s_i \to 1\).

When \(\Omega\) is a domain or a manifold, we first prove the decomposition by an extension argument on a ball 
and then extend the theorem to general domains and manifolds through local charts, a partition of the unity
and a suitable estimate on a low-frequency part that connects the local patches.
The regularity assumptions on the domain \(\Omega\) that we are making are probably not optimal in view of 
the possibility of extending functions under much weaker assumptions \cite{Zhou2015}.

\section{Nonlinear estimate of sums}

In this section, we prove that any sum of fractional Sobolev functions satisfies an estimate on a minimum.

\begin{proposition}
  \label{proposition_estimate_sum}
Let \(\ell \in \Nset^*\), let \(s_1, \dotsc, s_\ell \in (0, 1)\), let \(p_1, \dotsc, p_\ell \in [1, +\infty)\) and let \(\Omega\) be an \(n\)--dimensional Riemannian manifold with a possibly non-empty boundary.
If the functions \(u_1, \dotsc, u_\ell : \Omega \to \Rset\) are measurable and if \(u \defeq u_1 + \dotsb + u_\ell\),
then 
\[
  \iint\limits_{\Omega \times \Omega}
  \min_{1 \le i \le \ell} \frac{\abs{u (x)-u (y)}^{p_i}}{d_\Omega (x, y)^{n + s_i p_i}}
    \dif x \dif y
    \\ 
    \leq \iint\limits_{\Omega \times \Omega}
    \max_{1 \le i \le \ell} \ell^{p_i} \frac{\abs{u_i (x)-u_i (y)}^{p_i}}{d_\Omega (x, y)^{n + s_i p_i}}
    \dif x \dif y.    
\]
\end{proposition}

\begin{proof}
For every \(x, y \in \Omega\) such that \(x \ne y\), there exists \(j \in \{1, \dotsc, \ell\}\) such that 
\[
\abs{u_j (x) - u_j (y)} = \max_{1 \le i \le \ell} \, \abs{u_i (x) - u_i (y)},
\]
and thus by the triangle inequality, we have 
\[
\abs{u (x) - u (y)} \le 
\sum_{i = 1}^\ell \abs{u_i (x) - u_i (y)}
\le \ell \abs{u_j (x) - u_j (y)}.
\]
Therefore, for every \(x, y \in \Omega\),
\[
\frac{\abs{u (x) - u (y)}^{p_j}}
{d_\Omega (x, y)^{n + s_j p_j}}
\le \ell^{p_j} \frac{\abs{u_j (x) - u_j (y)}^{p_j}}
{d_\Omega (x, y)^{n + s_j p_j}}
\]
and thus 
\begin{equation}
  \label{eq_PheeZahr8ahy1goot1j}
\min_{1 \le i \le \ell}
\frac{\abs{u (x) - u (y)}^{p_i}}
{d_\Omega (x, y)^{n + s_i p_i}}
\le\max_{1 \le i \le \ell}\ell^{p_i} \frac{\abs{u_i (x) - u_i (y)}^{p_i}}
{d_\Omega (x, y)^{n + s_i p_i}} .
\end{equation}
The conclusion then follows by integrating the inequality \eqref{eq_PheeZahr8ahy1goot1j} with respect to \((x, y)\) over the set \(\Omega \times \Omega\).
\end{proof}

\begin{remark}
  \Cref{proposition_estimate_sum} can also be proved by Jensen's inequality applied to a suitable inf-convolution: one defines for each \(x, y \in \Omega\) the function \(\Phi_{x, y} : \Rset \to [0, +\infty)\) for every \(t \in \Rset\) by 
  \[
  \Phi_{x, y} (t) \defeq 
  \inf \biggl\{ \frac{\abs{t_1}^{p_1}}{d_\Omega (x, y)^{s_1 p_1}} + \dotsb +  \frac{\abs{t_\ell}^{p_\ell}}{d_\Omega (x, y)^{s_\ell p_\ell}}
  \st t_1, \dotsc, t_\ell \in \Rset \text{ and } t_1 + \dotsb + t_\ell = t\biggr\},
  \]
  one observes that the function \(\Phi_{x, y}\) is convex since \(p_1, \dotsc, p_\ell \ge 1\);
since \(u = \frac{1}{\ell} (\ell u_1 + \dotsb + \ell u_\ell)\), one has by Jensen's inequality
  \[
  \iint\limits_{\Omega \times \Omega} \frac{\Phi_{x, y} (u (x) - u (y))}{d_\Omega (x, y)^n} \dif x \dif y
  \le \sum_{i = 1}^\ell \frac{1}{\ell} \iint\limits_{\Omega \times \Omega} \frac{\Phi_{x, y} (\ell (u_i (x) - u_i (y)))}{d_\Omega (x, y)^n} 
  \dif x \dif y;
  \]
  one concludes by observing that for each \(x, y \in \Omega\) and \(t \in \Rset\), by definition of \(\Phi_{x, y}\),
  \[
  \min_{1 \le i \le \ell} \frac{\abs{t}^{p_i}}{\ell^{p_i}d_\Omega (x, y)^{s_ip_i}}
  \le 
  \Phi_{x, y} (t) \le \min_{1 \le i \le \ell} \frac{\abs{t}^{p_i}}{d_\Omega (x, y)^{s_ip_i}}.
  \]
\end{remark}

\section{Decomposition of functions in the Euclidean space}\label{sec:rn}

We decompose here measurable function on the Euclidean space \(\Rset^n\) with an estimate involving fractional Sobolev spaces.

\begin{proposition}
  \label{proposition_decomposition_Rn}
  Let \(n \in \Nset^*\), let \(\ell \in \Nset^*\), let \(s_1, \dotsc, s_\ell \in (0, 1)\) and let \(p_1, \dotsc, p_\ell \in [1, + \infty)\).
  There exists a constant \(C\) such that for every measurable function
  \(u : \Rset^n \to \Rset\) 
  there exist measurable functions \(u_1, \dotsc, u_\ell : \Rset^n \to \Rset\) such that \(u = u_1 + \dotsb + u_\ell\) almost everywhere on \(\Rset^n\) and 
  \[
  \iint\limits_{\Rset^n \times \Rset^n}
  \max_{1 \le i \le \ell}  \frac{\abs{u_i (x)-u_i (y)}^{p_i}}{\abs{x - y}^{n + s_i p_i}}
  \dif x \dif y
  \le 
  C
  \iint\limits_{\Rset^n \times \Rset^n}
  \min_{1 \le i \le \ell} \frac{\abs{u (x)-u (y)}^{p_i}}{\abs{x - y}^{n + s_i p_i}}
  \dif x \dif y
  .
  \]
\end{proposition}

Our first tool is the following Jensen type inequality for minima.

\begin{lemma}
  \label{lemma_Holder_min}
Let \(p_1, \dotsc, p_\ell \in [1, +\infty)\), let \(\mu\) be a probability measure on \(\Omega\) and let \(f : \Omega \to [0, +\infty)\). If \(f\) is \(\mu\)--measurable, then 
\[
\min_{1 \le i \le \ell} \alpha_i \, \biggl( \frac{1}{\ell} \int_{\Omega}  f \dif \mu \biggr)^{p_i}
\le 
\int_{\Omega} 
\min_{1 \le i \le \ell} \alpha_i f^{p_i}.
\]

\end{lemma}

\begin{proof}
  We define for every \(j \in \{1, \dotsc, \ell\}\), the set
\[
A_j = \bigl\{ x \in \Omega \st \alpha_j \abs{f (x)}^{p_j} = \min_{1 \le i \le \ell}
\alpha_i \abs{f (x)}^{p_i}\bigr\}.
\]
By definition, we have \(\bigcup_{j = 1}^{\ell} A_j = \Omega\), and thus there exists \(j \in \{1, \dotsc, \ell\}\) such that 
\[
\int_{\Omega} f \dif \mu
\le \ell  
\int_{A_j} f \dif \mu.
\]
Since \(\mu\) is a probability measure and \(p_j \ge 1\), we have by Jensen's inequality, 
\[
\Biggl(\int_{A_j} f \dif \mu \Biggr)^{p_j}
\le 
\int_{A_j} f^{p_j} \dif \mu,
\]
and thus 
\[
  \alpha_j \Biggl(\frac{1}{\ell} \int_{\Omega} f \dif \mu \Biggr)^{p_j}
  \le \alpha_j\Biggl( \int_{A_j} f \dif \mu\Biggr)^{p_j}
\le  \int_{A_j} \alpha_j f^{p_j} \le \int_{\Omega} \min_{1 \le i \le \ell} \alpha_i f^{p_i},
\]
and the conclusion follows.
\end{proof}

\begin{remark}
  \Cref{lemma_Holder_min} can also be proved by Jensen's inequality applied to a suitable inf-convolution:
  one defines the function \(\Phi : \Rset \to [0, +\infty)\) for each \(t \in \Rset\) by 
  \[
  \Phi (t) = 
  \,
  \inf \bigl\{ \alpha_1 \abs{t_1}^{p_1} + \dotsb + \alpha_\ell \abs{t_\ell}^{p_\ell}
  \st t_1, \dotsc, t_\ell \in \Rset \text{ and } t_1 + \dotsb + t_\ell = t\bigr\},
  \]
  one observes that the function \(\Phi\) is convex since \(p_1, \dotsc, p_\ell \ge 1\) and \(\alpha_1, \dotsc, \alpha_\ell \ge 0\);
  hence Jensen's inequality with \(\Phi\) and \(\mu\) applies to \(f\); one concludes by noting that by definition of \(\Phi\), for each \(t \in \Rset\),
  \[
  \min_{1 \le i \le \ell} \frac{\alpha_i \abs{t}^{p_i}}{\ell^{p_i}}
  \le 
  \Phi (t) \le \min_{1 \le i \le \ell} \alpha_i \abs{t}^{p_i}.
  \]
  
\end{remark}

Next we have an extension inequality with an estimate on a minimum of derivatives.

\begin{lemma}
  \label{lemma_extension_half_space}
  Let \(n \in \Nset^*\), let \(\ell \in \Nset^*\), let \(s_1, \dotsc, s_\ell \in (0, 1)\) and let \(p_1, \dotsc, p_\ell \in [1, + \infty)\).
  Assume that \(\varphi \in C^1_c (\Rset^n)\), \(\supp \varphi \subset B (0, 1)\) and \(\int_{\Rset^n} \varphi = 1\).
  There exists a constant \(C \in \Rset\) such that 
  for every \(u \in L^1_\mathrm{loc} (\Rset^n, \Rset)\), the function \(U : \Rset^{n + 1}_+ \to \Rset\) defined  for each \((x, t) \in \Rset^n \times (0, +\infty)
  \simeq \Rset^{n +1}_+\) by 
\[
  U (x, t) \defeq \int_{\Rset^n} u (x - th) \varphi (h) \dif h. 
\]
satisfies 
\[
\iint\limits_{\Rset^{n + 1}_+}
\min_{1 \le i \le \ell} \frac{\abs{\nabla U (x, t)}^{p_i}}{t^{1 - (1 - s) p_i}} \dif t \dif x
\le 
C
\iint\limits_{\Rset^n \times \Rset^n}
\min_{1 \le i \le \ell} \frac{\abs{u (x) - u (y)}^{p_i}}{\abs{x - y}^{n + s_i p_i}} \dif x \dif y.
\]
\end{lemma}

The function \(U\) is the convolution product of \(u\) with a family of rescaled functions, with the scaling parameter as the last variable. Indeed, one has for each \((x, t) \in \Rset^{n + 1}\),
\(
U (x, t) = (u\ast \varphi_t) (x)\), where for \(t \in (0, +\infty)\), the function \(\varphi_t : \Rset^n \to \Rset\) is defined for \(y \in \Rset^n\) by  \(\varphi_t (y) \defeq \frac{1}{t^n}\varphi (y/t)\).

The constant \(C\) in \cref{lemma_extension_half_space} remains bounded when \(\ell\) and \(p_1, \dotsc, p_\ell \in [1, +\infty)\) remain bounded.

\begin{proof}%
  [Proof of \cref{lemma_extension_half_space}]
  \resetconstant
  For every \((x, t) \in \Rset^{n +1}\), we have by a change of variable \(y = x - th\),
  \[
  U (x, t) = \frac{1}{t^n} \int_{\Rset^n} u (y) \varphi \Bigl(\frac{x - y}{t}\Bigr) \dif y.
  \]
    We define the function \(\xi \in C_c (\Rset^n, \Rset^{n + 1})\) for each \(x \in \Rset^n\) by \(\xi (x) \defeq 
  (\nabla \varphi (x), n \varphi (x) - \nabla \varphi (x) \cdot x )\), and we write 
  for every \((x, t) \in \Rset^{n + 1}_+\), 
  \[
  \begin{split}
   \nabla U (x, t) & =
   \frac{1}{t^{n + 1}} 
   \int_{\Rset^n} u (y) \,\xi \Bigl(\frac{x - y}{t}\Bigr) \dif y\\
   &=\frac{1}{t^{n + 1}} 
   \int_{\Rset^n} (u (y) - u (x)) \, \xi \Bigl(\frac{x - y}{t}\Bigr) \dif y,
 \end{split}
  \]
  since \(\int_{\Rset^n} \xi = 0\), and thus for each \((x, t) \in \Rset^{n + 1}_+\), 
  \[
  \abs{\nabla U (x, t)}
  \le \frac{\C}{t} 
  \fint_{B (x, t)} \abs{u (y) - u (x)} \dif y.
  \]

    We apply \cref{lemma_Holder_min} with \(\alpha_i \defeq 1/t^{1 - (1 - s_i)p_i}\), \(\mu\) the normalized Lebesgue measure on the ball \(B (x,t)\)  and the function \(f : \Rset^n \to \Rset\) defined for each \(y \in \Rset^n\) by \(f (y) \defeq \abs{u (y) - u (x)}\), and we get for each \((x, t) \in \Rset^n \times (0, +\infty)\),
  \begin{equation}
    \label{eq_aeb3Eefeenaizeih6ix}
    \min_{1 \le i \le \ell} \frac{\abs{\nabla U (x, t)}^{p_i}}{t^{1 - (1 - s_i)p_i}}
    \le
    \Cl{cst_oheegu6eik0Eesi7eoS}
    \int_{B (x, t)} \min_{1 \le i \le \ell} \frac{\abs{u (x) - u (y)}^{p_i}}{t^{n + s_i p_i + 1}}
    \dif y.
  \end{equation}
  By integration of the inequality \eqref{eq_aeb3Eefeenaizeih6ix} with respect to \(x\) and \(t\) and by Fubini's theorem, we get 
  \[
  \iint\limits_{\Rset^{n + 1}_+}
    \min_{1 \le i \le \ell} 
    \frac{\abs{\nabla U (x, t)}^{p_i}}{t^{1 - (1 - s_i)p_i}}
    \dif t \dif x
    \le 
    \Cr{cst_oheegu6eik0Eesi7eoS}
    \int_{\Rset^n} \int_{\Rset^n} \int_{\abs{x - y}}^{+\infty}
    \min_{1 \le i \le \ell} \frac{\abs{u (x) - u (y)}^{p_i}}{t^{n + s_i p_i + 1}}
    \dif t 
    \dif y
    \dif x.
  \]
  We conclude by \cref{lemma_int_min_powers} below.
\end{proof}

\begin{lemma}
  \label{lemma_int_min_powers}
  Let \(\gamma_1, \dotsc, \gamma_\ell \in (0, + \infty)\)
  and \(\beta_1, \dotsc, \beta_\ell \in [0, +\infty)\). For every \(r \in \Rset\), one has
  \[
  \int_{r}^{+\infty} \min_{1 \le i \le \ell}
  \frac{\beta_i}{t^{\gamma_i + 1}} \dif t
  \le \min_{1 \le i \le \ell} \frac{\beta_i}{\gamma_i r^{\gamma_i}}
  \]
  and 
  \[
  \int_{0}^{r} \min_{1 \le i \le \ell}
  {\beta_i}{t^{\gamma_i - 1}} \dif t
  \le \min_{1 \le i \le \ell} \frac{\beta_i r^{\gamma_i}}{\gamma_i}.
  \]
\end{lemma}
\begin{proof}
We fix \(r \in [0, +\infty)\) and we choose \(j \in \{1, \dotsc, \ell\}\) such that 
  \[
  \frac{\beta_j}{\gamma_j r^{\gamma_j}} = \min_{1 \le i \le \ell} \frac{\beta_i}{\gamma_i r^{\gamma_i}}.
  \]
We then have 
\[
\frac{\beta_j}{\gamma_j r^{\gamma_j}} = 
\int_{r}^{+\infty} \frac{\beta_j}{t^{\gamma_j + 1}} \dif t
\ge \int_r^{+\infty} \min_{1 \le i \le \ell}
\frac{\beta_i}{t^{\gamma_i + 1}} \dif t.
\]
The proof of the second inequality is similar.
\end{proof}

  From \cref{lemma_extension_half_space}, the function \(U\) can be decomposed as  \(\nabla U = \Theta_1 + \dotsb + \Theta_\ell\), with 
\[
\sum_{i = 1}^\ell
\iint\limits_{\Rset^{n + 1}_+}
\frac{\abs{\Theta_i (x, t)}^{p_i}}{t^{1 - (1 - s_i)p_i}} \dif t \dif x 
\le 
\iint\limits_{\Rset^{n + 1}_+}
\min_{1 \le i \le \ell} \frac{\abs{\nabla U(x,t)}^{p_i}}{t^{1 - (1 - s) p_i}} \dif t \dif x.
\]
(see \eqref{eq_reconstruction} and  \eqref{eq_decomposition_theta}  below).
In the following we show how to construct a function $u_i$ from the vector field $\Theta_i$ with an estimate of the Gagliardo semi-norms.

\begin{definition}
  \label{def_reconstruction}
The function \(\psi : \Rset^n \to \Rset^{n + 1}\) is a \emph{reconstruction kernel} whenever \(\psi \in C^1_c (\Rset^n, \Rset^{n + 1})\), \(\supp \psi \subset B (0, 1)\), for every \(x \in \Rset^n\)
\begin{equation*}
\operatorname{div} (\psi_x) (x)
  + x \cdot \nabla\psi_t (x) + n \psi_t (x) = 0
\end{equation*}
and 
\[
\int_{\Rset^n} \psi_t  = 1,
\]
where \(\psi = (\psi_x, \psi_t)\) with \(\psi_x : \Rset^n \to \Rset^n\) and \(\psi_t : \Rset^n \to \Rset\).
\end{definition}

For example, if \(\varphi \in C^1_c (\Rset^n)\), if \(\int_{\Rset^n} \varphi=1\) and if \(\supp \varphi \subset B (0, 1)\), then the function \(\psi : \Rset^n \to \Rset^{n + 1}\) defined for each \(x \in \Rset^n\) by \(\psi (x) \defeq \varphi (x)  (-x, 1)\) is a reconstruction kernel.

\begin{lemma}
\label{prop_reconstruction}
Let \(\psi :\Rset^n \to \Rset^{n + 1}\) be a reconstruction kernel and let \(U \in C^1 (\Rset^{n + 1}_+)\).
For every \(\tau < T\) and every \(x \in \Rset^n\),
\begin{multline*}
\frac{1}{\tau^n}\int_{\Rset^n} \psi_t \Bigl(\frac{x - y}{\tau}\Bigr) U (y, t) \dif y \\
  = 
  \frac{1}{T^n} \int_{\Rset^n} \psi_t \Bigl(\frac{x - y}{T}\Bigr) U (y, t) \dif y 
-
\iint\limits_{\Rset^n \times [\tau, T]}  \psi \Bigl(\frac{x - y}{t}\Bigr)\cdot \frac{\nabla U (y,t)}{t^{n}} \dif t \dif y.
\end{multline*}
\end{lemma}

\begin{proof}
Since the function \(U\) is smooth and \(\psi\) is compactly supported, we have by the divergence theorem 
\begin{multline}
  \iint\limits_{\Rset^{n}\times [\tau, T]}
  \biggl(
  \operatorname{div} (\psi_x) \Bigl(\frac{x - y}{t}\Bigr)+ \frac{x - y}{t} \cdot \nabla\psi_t \Bigl(\frac{x - y}{t}\Bigr) + n \psi_t  \Bigl(\frac{x -y}{t}\Bigr)
  \biggr) \frac{U (y, t)}{t^{n + 1}} \dif y \dif t\\
  -\iint\limits_{\Rset^{n}\times [\tau, T]} \psi \Bigl(\frac{x - y}{t}\Bigr) \cdot\frac{ \nabla U (y, t)}{t^n} \dif y \dif t \\
    =\frac{1}{\tau^n}\int_{\Rset^n} \psi_t \Bigl(\frac{x - y}{\tau}\Bigr) U (y, t) 
      - \frac{1}{T^n} \int_{\Rset^n} \psi_t \Bigl(\frac{x - y}{T}\Bigr) U (y, t) \dif y \dif t.
\end{multline}
The conclusion follows then from the definition of reconstruction kernel (\cref{def_reconstruction}).
\end{proof}

\begin{lemma}
  \label{prop_estimate_reconstruction}
  Let \(n \in \Nset^*\), let \(s \in (0, 1)\), let \(p \in [1, +\infty)\) and let \(\psi : \Rset^n \to \Rset^{n + 1}\) be a reconstruction kernel.
  There exists a constant \(C \in \Rset\) such that if the function \(\Theta : \Rset^{n + 1}_+ \to \Rset^{n + 1}\) is measurable and
  satisfies for almost every \(x \in \Rset^n\),
  \[
  \int_0^{+\infty} \int_{B (x, t)} \frac{\abs{\Theta (y, t)}}{t^{n}} \dif t \dif y
  < + \infty
  \]
and if the function \(v : \Rset^n \to \Rset\) is defined for almost every \(x \in \Rset^n\) by 
\begin{equation*}%\label{eq_defvarphi}
  v(x)\defeq\iint\limits_{\Rset_+^{n + 1}}
  \psi \Bigl(\frac{x - y}{t}\Bigr) \cdot \frac{\Theta (y, t)}{t^n} \dif t \dif y,
\end{equation*} 
then
\begin{equation*}
  \iint\limits_{\Rset^n \times \Rset^n}
  \frac{\abs{v (x) - v (y)}^p}{\abs{x - y}^{n + s p}}
  \dif x \dif y
  \leq 
  C
  \iint\limits_{\Rset^{n + 1}_+}  \frac{\abs{\Theta (x, t)}^p}{t^{1 - (1 - s) p}} \dif t \dif x 
  .
\end{equation*}
\end{lemma}

The constant \(C\) in \cref{prop_estimate_reconstruction} depends on the reconstruction kernel \(\psi\), on the \(s\) and \(p\) and blows up like $s^{-p} (1-s)^{-p}$ when \(s \to 0\) and \(s \rightarrow 1\).

Since by \cref{def_reconstruction}, \(\supp \psi \subset B (0, 1)\), the integrability assumption on the vector field \(\Theta\) ensures that the function \(v (x)\) is well-defined almost everywhere on \(\Rset^n\).

The proof of \cref{prop_estimate_reconstruction} follows the strategy of proofs of extensions of functions in fractional Sobolev spaces \citelist{\cite{Uspenskii}\cite{Mironescu_Russ2015}} and relies on the classical Hardy inequalities \cite{Hardy_Littlewood_Polya_1952}*{\S 329} (see also for example \cite{Mironescu_Russ2015}*{Proposition 2.1}).

\begin{lemma}
  \label{lemma_Hardy}
  Let $p\in [1,+\infty)$ and $\alpha \in (0,+\infty)$. 
  If the function $g : (0, +\infty) \to (0, + \infty)$ is measurable, then (Hardy inequality at \(0\))
  \begin{equation}\label{H0}
    \int_0^{+\infty} \left( \int_0^t g(r) \dif r \right)^p \frac{1}{t^{1 + \alpha}} \dif t  
    \leq \Bigl(\frac{p}{\alpha}\Bigr)^p  \int_0^{+\infty} \frac{g(r)^p}{r^{1 + \alpha - p}}  \dif r,
  \end{equation}
  and (Hardy inequality at \(\infty\))
  \begin{equation}\label{Hinfty}
    \int_0^{+\infty} \left( \int_t^{+\infty} g(r) \dif r \right)^p \frac{1}{t^{1 - \alpha}} \dif t  
    \leq \Bigl(\frac{p}{\alpha}\Bigr)^p  \int_0^{+\infty} \frac{g(r)^p}{r^{1 - \alpha - p}}  \dif r.
  \end{equation}
\end{lemma}

\begin{proof}
  [Proof of \cref{prop_estimate_reconstruction}]
  \resetconstant
By definition of the function \(v\), we have for every \(x, y \in \Rset^n\)
\begin{equation} 
  \label{eq_kai4hahriepie4ahNah}
 v (x) - v (y)
 = \iint\limits_{\Rset_+^{n + 1}}
 \biggl(\psi \Bigl(\frac{x - z}{t}\Bigr) - \psi \Bigl(\frac{y - z}{t}\Bigr) \biggr)
\cdot \frac{\Theta (z, t)}{t^n} \dif t \dif z.
\end{equation}
We next have by H\"older's inequality, for every \(x, y \in \Rset^n\),
\begin{equation*}
\Biggl\lvert\;
\iint\limits_{\Rset^{n} \times [0, \abs{x - y}]}
\psi \Bigl(\frac{x - z}{t}\Bigr)
\cdot \frac{\Theta (z, t)}{t^n}
\dif z \dif t \Biggr\rvert
\le 
\Cl{cst_Yahyahyaetohca5moay} \int_0^{\abs{x - y}} \biggl(\int_{B (x, t)} \frac{\abs{\Theta (z, t)}^p}{t^n} \dif z \biggr)^\frac{1}{p} \dif t .
\end{equation*}
Hence, by integration, 
\begin{multline*}
\iint\limits_{\Rset^n \times \Rset^n}
\Biggl\lvert
\;
\iint\limits_{\Rset^{n} \times [0, \abs{x - y}]}
\psi \Bigl(\frac{x - z}{t}\Bigr)
\cdot \frac{\Theta (z, t)}{t^n}
\dif z \dif t
\Biggr\rvert^p \frac{1}{\abs{x - y}^{n+sp}}
\dif x \dif y\\
\le 
\Cr{cst_Yahyahyaetohca5moay}
\iint\limits_{\Rset^n \times \Rset^n}
\biggl(
\int_0^{\abs{x - y}} \biggl(\int_{B (x, t)} \frac{\abs{\Theta (z, t)}^p}{t^{n}} \dif z \biggr)^\frac{1}{p} \dif t
\biggr)^p \frac{1}{\abs{x - y}^{n + sp}} \dif x \dif y
.
\end{multline*}
By performing the integration in spherical coordinates of \(y\)  centred at \(x\), we get 
\begin{multline*}
  \iint\limits_{\Rset^n \times \Rset^n}
  \Biggl\lvert
  \;
  \iint\limits_{\Rset^{n} \times [0, \abs{x - y}]}
  \psi \Bigl(\frac{x - z}{t}\Bigr)
  \cdot \frac{\Theta (z, t)}{t^n}
  \dif t \dif z
  \Biggr\rvert^p \frac{1}{\abs{x - y}^{n+sp}}
  \dif x \dif y\\
  \le 
  \C 
  \int_{\Rset^n}
  \int_0^{+\infty} 
  \biggl(
  \int_0^{r} \biggl(\int_{B (x, t)} \frac{\abs{\Theta (z, t)}^p}{t^{n}} \dif z \biggr)^\frac{1}{p}  \dif t
  \biggr)^p \frac{1}{r^{1 + sp}} \dif r \dif x.
\end{multline*}
In view of Hardy's inequality at \(0\) (\cref{lemma_Hardy}), we have for every \(x \in \Rset^n\), 
\begin{multline*}
  \int_0^{+\infty} 
  \biggl( 
  \int_0^{r} \biggl(\int_{B (x, t)} \frac{\abs{\Theta (z, t)}^p}{t^{n}} \dif z \biggr)^\frac{1}{p}  \dif t
\biggr)^p \frac{1}{r^{1 + sp}} \dif r\\
\le 
\frac{1}{s^p} 
\int_0^{+\infty} 
\int_{B (x, t)} \frac{\abs{\Theta (z, t)}^p}{t^{n + 1 - (1 - s) p}} \dif z \dif t.
\end{multline*}
Hence, we have 
\begin{multline}
  \label{eq_ohGie2oovi8thougied}
  \iint\limits_{\Rset^n \times \Rset^n}
  \Biggl\lvert
  \;
  \iint\limits_{\Rset^{n} \times [0, \abs{x - y}]}
  \psi \Bigl(\frac{x - z}{t}\Bigr)
  \cdot \frac{\Theta (z, t)}{t^n}
  \dif t \dif z
  \Biggr\rvert^p \frac{1}{\abs{x - y}^{n+sp}}
  \dif x \dif y\\
  \le 
  \C 
  \int_{\Rset^n}
  \int_0^{+\infty} 
  \int_{B (x, t)} \frac{\abs{\Theta (z, t)}^p}{t^{n + 1 - (1 - s) p}} \dif z \dif t \dif x
  \le \Cl{cst_Ahyaib5aiD7oSoR0aFa} 
  \int_{\Rset^{n + 1}_+}
    \frac{\abs{\Theta (z, t)}^p}{t^{1 - (1 - s) p}} \dif z \dif t.
\end{multline}
Similarly, we have by exchanging \(x\) and \(y\),
\begin{equation}
  \label{eq_Ahyaib5aiD7oSoR0aFa}
  \iint\limits_{\Rset^n \times \Rset^n}
  \Biggl\lvert
  \;
  \iint\limits_{\Rset^{n} \times [0, \abs{x - y}]}
  \psi \Bigl(\frac{y - z}{t}\Bigr)
  \cdot \frac{\Theta (z, t)}{t^n}
  \dif t \dif z
  \Biggr\rvert^p \frac{1}{\abs{x - y}^{n+sp}}
  \dif x \dif y
  \le \Cr{cst_Ahyaib5aiD7oSoR0aFa} 
  \iint\limits_{\Rset^{n + 1}_+}
  \frac{\abs{\Theta (z, t)}^p}{t^{1 - (1 - s) p}} \dif z \dif t.
\end{equation}
We observe now that if \(t \ge \abs{x - y}\), then 
\[
B (x, t) \cup B (y, t) \subset B \Bigl(\frac{x + y}{2}, t + \frac{\abs{x - y}}{2} \Bigr) 
\subset B \Bigl(\frac{x + y}{2}, \frac{3t}{2} \Bigr).
\]
Moreover since the function \(\psi\) is Lipschitz continuous, we have for every \(x, y, z \in \Rset^n\) and every \(t \in (0, +\infty)\),
\[
\Bigl\lvert
\psi \Bigl(\frac{x - z}{t}\Bigr) - \psi \Bigl(\frac{y - z}{t}\Bigr)
\Bigr\rvert
\le \C \frac{\abs{x - y}}{t}.
\]

We have thus 
\begin{multline*}
  \Biggl\lvert
  \;
  \iint\limits_{\Rset^{n}\times [\abs{x - y}, +\infty)}
\biggl(\psi \Bigl(\frac{x - z}{t}\Bigr) - \psi \Bigl(\frac{y - z}{t}\Bigr) \biggr)
\cdot \frac{\Theta (z, t)}{t^n} \dif t \dif z 
\Biggr\rvert \\
\le \C\, \abs{x - y} 
\int_{\abs{x - y}}^{+\infty} \int_{B (\frac{x + y}{2}, \frac{3t}{2})} 
\frac{\abs{\Theta (z, t)}}{t^{n + 1}} \dif z \dif t 
\end{multline*}
By H\"older's inequality, we deduce that 
\begin{multline*}
\Biggl\lvert
\;
\iint\limits_{\Rset^{n}\times [\abs{x - y}, +\infty)}
\biggl(\psi \Bigl(\frac{x - z}{t}\Bigr) - \psi \Bigl(\frac{y - z}{t}\Bigr) \biggr)
\cdot \frac{\Theta (z, t)}{t^n} \dif t \dif z 
\Biggr\rvert\\
\le \C\, \abs{x - y}
\int_{\abs{x - y}}^{+\infty} \biggl(\int_{B (\frac{x + y}{2}, \frac{3t}{2})} 
\frac{\abs{\Theta (z, t)}^p}{t^{n + p}}\dif z \biggr)^\frac{1}{p} \dif t .
\end{multline*}
By integration with respect to \((x, y)\) over \(\Rset^n \times \Rset^n\), we get 
\begin{multline*}
  \iint\limits_{\Rset^n \times \Rset^n}
  \Biggl\lvert
  \;
  \iint\limits_{\Rset^{n}\times [\abs{x - y}, +\infty)}
  \biggl(\psi \Bigl(\frac{x - z}{t}\Bigr) - \psi \Bigl(\frac{y - z}{t}\Bigr) \biggr)
  \cdot \frac{\Theta (z, t)}{t^n} \dif t \dif z 
  \Biggr\rvert^p \frac{1}{\abs{x - y}^{n + sp}} \dif x \dif y 
  \\
  \le \C
  \iint\limits_{\Rset^n \times \Rset^n}
  \biggl(\int_{\abs{x - y}}^{+\infty} \biggl(\int_{B (\frac{x + y}{2}, \frac{3t}{2})} 
  \frac{\abs{\Theta (z, t)}^p}{t^{n + p}}\dif z \biggr)^\frac{1}{p} \dif t
  \biggr)^p \frac{1}{\abs{x - y}^{n - (1 - s)p}} \dif x \dif y.
\end{multline*}
By a change of variable \(x = w + r \sigma\) and \(y = w - r \sigma\), with \(w \in \Rset^n\), \(r \in (0, + \infty)\) and \(\sigma \in \mathbb{S}^{n - 1}\), we get 
\begin{multline*}
  \iint\limits_{\Rset^n \times \Rset^n}
  \Biggl\lvert\;
  \iint\limits_{\Rset^{n}\times [\abs{x - y}, +\infty)}
  \biggl(\psi \Bigl(\frac{x - z}{t}\Bigr) - \psi \Bigl(\frac{y - z}{t}\Bigr) \biggr)
  \cdot \frac{\Theta (z, t)}{t^n} \dif t \dif z 
  \Biggr\rvert^p \frac{1}{\abs{x - y}^{n + sp}} \dif x \dif y 
  \\
  \le \C
  \int_{\Rset^n}
  \int_0^{+\infty}
  \biggl(\int_{r}^{+\infty} \biggl(\int_{B (w, \frac{3t}{2})} 
  \frac{\abs{\Theta (z, t)}^p}{t^{n + p}}\dif z \biggr)^\frac{1}{p} \dif t
  \biggr)^p \frac{1}{r^{1 - (1 - s)p}} \dif r \dif w.
\end{multline*}
  By Hardy's inequality at \(\infty\) (\cref{lemma_Hardy}), we get for every \(w \in \Rset^n\), 
\begin{multline*}
\int_0^{+\infty}
\biggl(\int_{r}^{+\infty} \biggl(\int_{B (w, \frac{3t}{2})} 
\frac{\abs{\Theta (z, t)}^p}{t^{n + p}}\dif z \biggr)^\frac{1}{p} \dif t
\biggr)^p \frac{1}{r^{1 - (1 - s)p}} \dif r\\
\le \frac{1}{(1 - s)^p} \int_0^{+\infty} \int_{B (w, \frac{3t}{2})} 
\frac{\abs{\Theta (z, t)}^p}{t^{n + 1 - (1 - s)p}} \dif z \dif t,
\end{multline*}
and thus 
\begin{multline}
  \label{eq_seg9Oowohngeishooqu}
  \iint\limits_{\Rset^n \times \Rset^n}
  \Biggl\lvert\;
  \iint\limits_{\Rset^{n}\times [\abs{x - y}, +\infty)}
  \biggl(\psi \Bigl(\frac{x - z}{t}\Bigr) - \psi \Bigl(\frac{y - z}{t}\Bigr) \biggr)
  \cdot \frac{\Theta (z, t)}{t^n} \dif t \dif z 
  \Biggr\rvert^p \frac{1}{\abs{x - y}^{n + sp}} \dif x \dif y 
  \\
  \le 
  \C
  \int_{\Rset^n}
  \int_0^{+\infty} \int_{B (w, \frac{3t}{2})} 
  \frac{\abs{\Theta (z, t)}^p}{t^{n + 1 - (1 - s)p}} \dif z \dif t \dif w
  \le \C
  \iint\limits_{\Rset^{n + 1}_+}
  \frac{\abs{\Theta (z, t)}^p}{t^{1 - (1 - s)p}} \dif z \dif t.
\end{multline}
By combining the inequalities \eqref{eq_kai4hahriepie4ahNah}, \eqref{eq_ohGie2oovi8thougied}, \eqref{eq_Ahyaib5aiD7oSoR0aFa} and \eqref{eq_seg9Oowohngeishooqu}, we reach the conclusion. 
\end{proof}

\begin{proof}[Proof of \cref{proposition_decomposition_Rn}  when \(
  \int_{\Rset^n} \abs{u (x)}/(1 + \abs{x}^n) \dif x < + \infty\)
  ]
  \resetconstant
  We fix a function \(\varphi : \Rset^n \to \Rset\) that satisfies the assumptions of \cref{lemma_extension_half_space}
  and a reconstruction kernel \(\psi : \Rset^n \to \Rset^{n + 1}\).
Let \(U \in C^1 (\Rset^{n + 1}_+)\) be the function defined in \cref{lemma_extension_half_space}.
Since the function \(u\) is locally integrable, for almost every \(x \in \Rset^n\), we have \(u(x)=\lim_{t\rightarrow 0} U(x,t)\). 
By letting \(\tau \rightarrow 0\) and  \(T \rightarrow +\infty\) in \cref{prop_reconstruction} and noting that 
\(
  \liminf_{R \to \infty} \frac{1}{R^n} \int_{B (0, R)} \abs{u} = 0\), we obtain for almost every \(x \in \Rset^n\)
\begin{equation}
  \label{eq_reconstruction}
  u (x) = - \iint\limits_{\Rset^{n + 1}_+} \psi \Bigl(\frac{y-x}{t}\Bigr)\cdot \frac{\nabla U (y,t)}{t^{n}} \dif t \dif y.
\end{equation}

We define for each \(i \in \{1, \dotsc, \ell\}\), the vector field \(\Theta_i : \Rset^{n + 1}_+ \to \Rset^{n + 1}\) by
\begin{equation}
  \label{eq_decomposition_theta}
  \Theta_i 
  \defeq
  \textbf{1}_{A_i} \nabla U ,
\end{equation}
where the function \(\textbf{1}_{A_i} : \Rset^{n + 1}_+ \to \Rset\) is the characteristic function of the set 
\begin{multline*}
A_i \defeq \Bigl\{(x, t) \in \Rset^n   \times \Rset^+ \st
\text{for each \(j \in \{1, \dotsc, i - 1\}\) }
\frac{\abs{\nabla U (x, t)}^{p_i}}{t^{1 - (1 - s_i)p_i}}
< \frac{\abs{\nabla U (x, t)}^{p_j}}{t^{1 - (1 - s_j)p_j}}\\
\text{ and for each \(j \in \{i + 1, \dotsc,\ell \}\) }
\frac{\abs{\nabla U (x, t)}^{p_i}}{t^{1 - (1 - s_i)p_i}}
\le \frac{\abs{\nabla U (x, t)}^{p_j}}{t^{1 - (1 - s_j)p_j}}
\Bigr\}.
\end{multline*}
We observe that \(\bigcup_{i = 1}^\ell A_i = \Rset^n\) and that if \(i, j \in \{1, \dotsc, \ell\}\) and \(i \ne j\), one has \(A_i \cap A_j= \emptyset\).
Therefore,
\begin{equation}
  \label{eq_aesh8nie1iexaew8Fah}
  \Theta_1  + \dots + \Theta_\ell = \nabla U
\end{equation}
and by \cref{lemma_extension_half_space}
\begin{equation}
  \label{eq_ooSh6aeN2aiCh4om5xi}
  \begin{split}
\sum_{i = 1}^\ell 
\iint\limits_{\Rset^{n + 1}_+} \frac{\abs{\Theta_i (x, t)}^p}{t^{1 - (1 - s_i)p_i}}\dif x
\dif t  
&= \sum_{i = 1}^\ell 
\iint\limits_{A_i} \frac{\abs{\nabla U (x, t)}^p}{t^{1 - (1 - s_i)p_i}}\dif x
\dif t  \\
&= \iint\limits_{\Rset^{n+1}_+} \min_{1\leq i \leq \ell} \frac{|\nabla U(x,t)|^{p_i}}{t^{1 - (1 - s_i)p_i}} \dif x \dif t\\
&\le \C \iint\limits_{\Rset^n \times \Rset^n}
\min_{1 \le i \le \ell} \frac{\abs{u (x) - u (y)}^{p_i}}{\abs{x - y}^{n + s_i p_i}} \dif x \dif y.
\end{split}
\end{equation}

For every \(i \in \{1, \dotsc, \ell\}\) and \(R \in (0, +\infty)\), we have by H\"older's inequality,
\begin{equation}
  \label{eq_ahgan8Ain5ahNao2ael}
\begin{split}
  \int_{B (0, R)} 
  &
    \int_0^{R} 
      \int_{B(x, t)} 
        \frac
          {\abs{\Theta_i (y, t)}}
          {t^{n}} 
        \dif y 
        \dif t
      \dif x \\
 &= 
 \int_{B (0, 2R)}
  \int_0^R 
  \int_{B(y, t) \cap B (0, R)} 
  \frac
  {\abs{\Theta_i (y, t)}}
  {t^{n}} 
  \dif x
  \dif t 
  \dif y\\
& \le
\C 
  \iint\limits_{B (0, 2R) \times (0, R)}
  \abs{\Theta_i (y, t)}
  \dif y 
  \dif t\\
&\le \C
\biggl(
\iint\limits_{B (0, 2R) \times (0, R)}
\frac
{\abs{\Theta_i (y, t)}^{p_i}}
{t^{1 - (1 - s_i)p_i}} 
\dif y \dif t
\biggr)^\frac{1}{p_i}
\biggl( 
\iint\limits_{B (0, 2R) \times (0, R)}
\frac
{1}
{t^{1 - s_i \frac{p_i}{p_i - 1}}} 
\dif y \dif t
\biggr)^{1 - \frac{1}{p_i}}\\
&< +\infty.
  \end{split}
\end{equation}
On the other hand, we have for each \(x \in B (0, R)\), 
\begin{equation}\label{eq_Aoy2ezeib5yaiveeyee}
\begin{split}
\int_{R}^{+\infty} \int_{B (x, t)} \frac{\abs{\Theta_i (y, t)}}{t^n} \dif y\dif t
&\le \int_{R}^{+\infty} \int_{B (x, t)} \frac{\abs{\nabla U (y, t)}}{t^n} \dif y \dif t\\
&\le \C \int_{R}^{+\infty} \int_{B (x, t)} \int_{B (y, t)} \frac{\abs{u (z)}}{t^{2 n + 1}} \dif z \dif y \dif t\\
&\le \C \int_{R}^{+\infty} \int_{B (x, 2t)} \frac{\abs{u (z)}}{t^{n + 1}} \dif z \dif t\\
&\le \C \int_{\Rset^n} \frac{\abs{u (z)}}{(R + \abs{z - x})^n} \dif z < +\infty.
\end{split}
\end{equation}
Hence, by \eqref{eq_ahgan8Ain5ahNao2ael} and \eqref{eq_Aoy2ezeib5yaiveeyee}, for almost every \(x \in \Rset^n\),
\begin{equation}
  \label{eq_Yoos9Iefiefooc5peen}
\int_0^{+\infty} 
\int_{B(x, t)} 
\frac
{\abs{\Theta_i (y, t)}}
{t^{n}} 
\dif y 
\dif t
< 
+\infty .
\end{equation}
In view of \eqref{eq_Yoos9Iefiefooc5peen}, we define for each \(i \in \{1, \dotsc, \ell\}\) the function \(u_i : \Rset^n \to \Rset\) by setting 
for each \(x \in \Rset^n\) 
\[
u_i (x) \defeq - \iint\limits_{\Rset^{n + 1}_+}  \psi \Bigl(\frac{y-x}{t}\Bigr)\cdot \frac{\Theta_i (y, t)}{t^{n}} \dif t \dif y.
\]
In view of \eqref{eq_aesh8nie1iexaew8Fah} and \eqref{eq_reconstruction}, we have 
\[
 u = u_1 + \dotsb + u_\ell
\]
almost everywhere in \(\Rset^n\).
Moreover, by \cref{prop_estimate_reconstruction}, we have
\begin{equation*}
  \iint\limits_{\Rset^n \times \Rset^n}
  \frac{\abs{u_i (x) - u_i (y)}^p}{\abs{x - y}^{m + s p}}
  \dif x \dif y
  \leq 
  \C
  \iint\limits_{\Rset^{n + 1}_+}  \frac{\abs{\Theta_i (x, t)}^p}{t^{1 - (1 - s) p}} \dif t \dif x .
\end{equation*}
We conclude by the estimate \eqref{eq_ooSh6aeN2aiCh4om5xi}. 
\end{proof}

\begin{remark}
  When \(p_1, \dotsc, p_\ell > 1\), the functions \(u_1, \dotsc, u_\ell\) can also be constructed 
by estimates on the Riesz transform with Muckenhoupt weights and trace theory.
One extend \(\Theta_i\) to \(\Rset^{n + 1}\) in such a way that 
\(\Theta_i : \Rset^{n + 1} \to \Rset^{n + 1}\) commutes with the reflection with respect to the hyperplane \(\Rset^n \times \{0\}\).
  One defines then \(\Xi_i \defeq \mathcal{R} \mathcal{R} \cdot \Theta_i\), where \(\mathcal{R}
  = (\mathcal{R}_1, \dotsc, \mathcal{R}_{n + 1})\) is the vector Riesz transform.
  The weights appearing in \eqref{eq_ooSh6aeN2aiCh4om5xi} satisfy the Muckenhoupt condition and  thus \(\Xi_i\) satisfies an estimate with the same weight 
  \citelist{ \cite{Stein1993}{Theorem V.2}\cite{Coifman_Fefferman_1974}*{Theorem III}}.
  By construction, there exists a function \(U_i : \Rset^{n + 1} \to \Rset\) such that 
  \(\nabla U_i = \Xi_i\). One defines then \(u_i\) to be the trace of \(U_i\).
  One has \(u = u_1 + \dotsb + u_\ell\) because \(\nabla U = \mathcal{R} \mathcal{R} \cdot \nabla U\).
\end{remark}

\begin{remark}
  \resetconstant
  If for every \(i \in \{1, \dotsc \ell\}\) one has \(n > s_i p_i\), then the condition 
  \[
  \iint\limits_{\Rset^n \times \Rset^n}
  \min_{1 \le i \le \ell} \frac{\abs{u (x)-u (y)}^{p_i}}{\abs{x - y}^{n + s_i p_i}}
  \dif x \dif y
  < + \infty,
  \]
  implies the existence of a constant \(\kappa \in \Rset\) such that 
  \[
  \int_{\Rset^n} \frac{\abs{u (x) - \kappa}}{1 + \abs{x}^n} \dif x  <+\infty.
  \]  
Indeed, for every \(R \in (0, +\infty)\), one has by \cref{lemma_Holder_min},
\[
\min_{1 \le i \le \ell} 
\biggl( R^{\frac{n}{p_i} - s_i} \fint\limits_{B (0,2R)}\fint\limits_{B (0, R)} \abs{u (x) -  u(y)} \dif x \dif y \biggr)^{p_i}
\le 
\C \iint\limits_{\Rset^n \times \Rset^n}
\min_{1 \le i \le \ell} \frac{\abs{u (x)-u (y)}^{p_i}}{\abs{x - y}^{n + s_i p_i}}
\dif x \dif y;
\]
if \(\alpha \defeq \min_{1 \le i \le \ell} \frac{n}{p_i}-s_i  > 0\), one deduces that when \(R \in (1, +\infty)\)
\[
  R^{\alpha} \fint_{B (0, 2R)}\fint_{B (0, R)} \abs{u (x) -  u(y)} \dif x \dif y\le \C;
\]
and hence by a dyadic decomposition of radii, if \(\rho \ge R \ge 1\),
\[
R^{\alpha} \fint_{B (0, \rho)}\fint_{B (0, R)} \abs{u (x) -  u(y)} \dif x \dif y \le \C;
\]
hence 
\[
\kappa = \lim_{\rho \to \infty} \fint_{B (0, \rho)} u,
\]
is well-defined, and 
\[
\begin{split}
\int_{\Rset^n} \frac{\abs{u (x) - \kappa}}{1 + \abs{x}^n} \dif x
&\le 
\C \int_1^{+\infty} \frac{1}{r^{n + 1}} \biggl(\int_{B (0, r)} \abs{u - \kappa} \biggr) \dif r\\
&\le 
\C \int_1^{+\infty} \frac{1}{r^{\alpha + 1}} \dif r < +\infty.
  \end{split}
\]
This approach fails when \(\max_{1 \le i \le n} s_i p_i \ge n\) since there exist then functions \(u \in \dot{W}^{s_1, p_1} (\Rset^n)\) with \(sp \geq n \) such that
\(\lim_{\abs{x} \to +\infty} \abs{u (x)} = + \infty\). 
\end{remark}

In order to treat the case where \(u \in L^1_{\mathrm{loc}} (\Rset^n)\) but \(\int_{\Rset^n} \abs{u (x)}/(1+ \abs{x}^n) \dif x =+\infty\) we rely on a truncation construction.

\begin{lemma}
\label{lemma_cut_domain}
Let \(n \in \Nset^*\), let \(\ell \in \Nset^*\), let \(s_1, \dotsc, s_\ell \in (0, 1)\) and let \(p_1, \dotsc, p_\ell \in [1,+\infty)\).
There exists a constant \(C\) such that if \(R \in (0, +\infty)\), if \(u \in L^1 (B (0, R))\), if 
\[
\int_{B (0, R)} u = 0,
\]
and if \(\eta \in C^{0, 1}_c (B (0, R))\), then 
\begin{multline*}
  \iint\limits_{\Rset^n \times \Rset^n}
  \min_{1 \le i \le \ell} \frac{\abs{(\eta u) (x)- (\eta u) (y)}^{p_i}}{\abs{x - y}^{n + s_i p_i}}
  \dif x \dif y\\
  \le 
  C 
  \iint\limits_{B (0, R) \times B (0, R)}
  \min_{1 \le i \le \ell} \frac{(R \abs{\eta}_{C^{0, 1}} \abs{u (x)- u (y)})^{p_i}}{\abs{x - y}^{n + s_i p_i}}
  \dif x \dif y.
\end{multline*}
\end{lemma}

In the statement of \cref{lemma_cut_domain}, we define \((u \eta)= 0\) on \(\Rset^n \setminus B (0, R)\).

\begin{proof}%
[Proof of \cref{lemma_cut_domain}]%
\resetconstant
We have for every \(x \in B (0, R)\) and \(y \in \Rset^n\),
\begin{equation}
  \label{eq_Lusaiph9ochaequioch}
\abs{ \eta (x) u (x) - \eta (y) u (y)}
\le \abs{\eta (x) - \eta (y)} \abs{u (x)} 
+ \abs{\eta (y)} \abs{u (x) - u (y)}.
\end{equation}
We define the sets 
\[
A = 
\bigl\{( x, y) \in B (0, R) \times \Rset^n \st \abs{ \eta (x) u (x) - \eta (y) u (y)}
\le 2 \abs{\eta (y)} \abs{u (x) - u (y)} \bigr\}
\]
and \[
 B = 
 \bigl\{( x, y) \in B (0, R) \times \Rset^n \st \abs{ \eta (x) u (x) - \eta (y) u (y)}
 \le 2 \abs{\eta (x) - \eta (y)} \abs{u (x)} \bigr\}.
\]

In view of \eqref{eq_Lusaiph9ochaequioch} we have \(A \cup B = B (0, R) \times \Rset^n\), and therefore 
\begin{equation}
  \label{eq_quieChei4wu9ienae}
\begin{split}
\iint\limits_{B (0, R) \times \Rset^n}&
\min_{1 \le i \le \ell} \frac{\abs{(\eta u) (x)- (\eta u) (y)}^{p_i}}{\abs{x - y}^{n + s_i p_i}}
\dif x \dif y\\
&\le 
\iint\limits_{A}
\min_{1 \le i \le \ell} \frac{(2 \abs{\eta (y)} \abs{u (x) - u (y)})^{p_i}}{\abs{x - y}^{n + s_i p_i}}
\dif x \dif y\\
&\qquad +
\iint\limits_{B}
\min_{1 \le i \le \ell} \frac{(2 \abs{\eta (x) - \eta (y)} \abs{u (x)})^{p_i}}{\abs{x - y}^{n + s_i p_i}}
\dif x \dif y.
\end{split}
\end{equation}
We first observe that since \(\eta = 0\) in \(\Rset^n \setminus B (0, R)\), we have 
\(\abs{\eta} \le \abs{\eta}_{C^{0, 1}} R\) in \(B (0, R)\), and thus 
\begin{multline}  
  \label{eq_oite1ohH6neingeuw1o}
  \iint\limits_{A}
  \min_{1 \le i \le \ell} \frac{(2 \abs{\eta (y)} \abs{u (x) - u (y)})^{p_i}}{\abs{x - y}^{n + s_i p_i}}
  \dif x \dif y\\
  \le 
  \Cl{cst_aer2va8aishieKeex4o}
\iint\limits_{B (0, R) \times \Rset^n}
\min_{1 \le i \le \ell} \frac{(\abs{\eta (y)} \abs{u (x) - u (y)})^{p_i}}{\abs{x - y}^{n + s_i p_i}}
\dif x \dif y\\
\le \Cr{cst_aer2va8aishieKeex4o}\iint\limits_{B (0, R) \times B (0, R)}
\min_{1 \le i \le \ell} \frac{(\abs{\eta}_{C^{0, 1}} R \abs{u (x) - u (y)})^{p_i}}{\abs{x - y}^{n + s_i p_i}}
\dif x \dif y.
\end{multline}
Next we observe that since \(\int_{B (0, R)} u = 0\) by our assumption, we have 
for every \(x \in B (0, R)\),
\[
  \abs{u (x)}
  \le \fint_{B (0, R)} \abs{u (x) - u (z)} \dif z,
\]
and thus  by \cref{lemma_Holder_min}, for every \(x \in B (0, R)\) and \(y \in \Rset^n\), 
\[
\min_{1 \le i \le \ell} \frac{(\abs{\eta (x) - \eta (y)} \abs{u (x)})^{p_i}}{\abs{x - y}^{n + s_i p_i}}
\le
\C
\fint_{B (0, R)} \min_{1 \le i \le \ell} \frac{(\abs{\eta (x) - \eta (y)} \abs{u (x) - u (z)})^{p_i}}{\abs{x - y}^{n + s_i p_i}}
\dif z .
\]
Since for every \(x, y \in \Rset^n\), \(\abs{\eta (x) - \eta (y)} \le \abs{\eta}_{C^{0, 1}} \min (\abs{x -y}, R)\), we deduce that 
\begin{multline*}\iint\limits_{B}
  \min_{1 \le i \le \ell} \frac{(2 \abs{\eta (x) - \eta (y)} \abs{u (x)})^{p_i}}{\abs{x - y}^{n + s_i p_i}}
  \dif x \dif y\\
  \shoveleft{\le \C
\iint\limits_{B (0, R) \times \Rset^n} 
\min_{1 \le i \le \ell} \frac{(\abs{\eta (x) - \eta (y)} \abs{u (x)})^{p_i}}{\abs{x - y}^{n + s_i p_i}}\dif x \dif y }\\
\le
  \frac{\C}{R^n} 
  \iint\limits_{B(0, R) \times B (0, R)}
  \int\limits_{B (x, R)} \min_{1 \le i \le \ell} \frac{(\abs{\eta}_{C^{0, 1}} \abs{u (x) - u (z)})^{p_i}}{\abs{x - y}^{n - (1 - s_i) p_i}}\dif y \dif x \dif z\\
+ 
  \frac{\C}{R^n} 
  \iint\limits_{B(0, R) \times B (0, R)} 
  \int\limits_{\Rset^n \setminus B (x, R)} \min_{1 \le i \le \ell} \frac{(\abs{\eta}_{C^{0, 1}} R \abs{u (x) - u (z)})^{p_i}}{\abs{x - y}^{n + s_i p_i}}\dif y \dif x \dif z. 
\end{multline*}
We rewrite the integrals  with respect to \(y\) on the right-hand cite in spherical coordinates centred 
at the point \(x\) and we obtain
\begin{multline*}
  \iint\limits_{B} 
  \min_{1 \le i \le \ell} \frac{(\abs{\eta (x) - \eta (y)} \abs{u (x)})^{p_i}}{\abs{x - y}^{n + s_i p_i}}\dif x \dif y \\
  \le
  \frac{\C}{R^n} \iint\limits_{B(0, R) \times B (0, R)} \int_{0}^{R} \min_{1 \le i \le \ell} \frac{(\abs{\eta}_{C^{0, 1}} \abs{u (x) - u (z)})^{p_i}}{r^{1 - (1 - s_i) p_i}}\dif r \dif x \dif z\\
  + \frac{\C}{R^n} \iint\limits_{B(0, R) \times B (0, R)} \int_{R}^{+\infty} \min_{1 \le i \le \ell} \frac{(\abs{\eta}_{C^{0, 1}} R \abs{u (x) - u (z)})^{p_i}}{r^{1 + s_i p_i}}\dif r \dif x \dif z. 
\end{multline*}
By applying \cref{lemma_int_min_powers}, we deduce that 
\begin{multline}
  \label{eq_IiCioc7een4vaegeque}
  \iint\limits_{B}
   \min_{1 \le i \le \ell} \frac{(\abs{\eta (x) - \eta (y)} \abs{u (x)})^{p_i}}{\abs{x - y}^{n + s_i p_i}}\dif x \dif y\\
   \le
   \C \iint\limits_{B(0, R) \times B (0, R)} \min_{1 \le i \le \ell} \frac{(\abs{\eta}_{C^{0, 1}} R \abs{u (x) - u (z)})^{p_i}}{R^{n + s_i p_i}} \dif x \dif z\\
   \le 
   \C \iint\limits_{B(0, R) \times B (0, R)} \min_{1 \le i \le \ell} \frac{(\abs{\eta}_{C^{0, 1}} R \abs{u (x) - u (z)})^{p_i}}{\abs{x - z}^{n + s_i p_i}} \dif x \dif z.
\end{multline}
The conclusion follows from \eqref{eq_quieChei4wu9ienae}, \eqref{eq_oite1ohH6neingeuw1o} and \eqref{eq_IiCioc7een4vaegeque}.
\end{proof}

\begin{proof}[Proof of \cref{proposition_decomposition_Rn} in the general case]
  \resetconstant
  We assume without loss of generality that 
  \[
 \iint\limits_{\Rset^n \times \Rset^n}
  \min_{1 \le i \le \ell} \frac{\abs{u (x)-u (y)}^{p_i}}{\abs{x - y}^{n + s_i p_i}}
  \dif x \dif y < + \infty;
  \]  
  it follows then that \(u \in L^1_\mathrm{loc} (\Rset^n)\) (see \cref{lemma_integrability} below).
  We choose a function \(\eta \in C^1_c (\Rset^n)\) such that \(\eta = 0\) in \(\Rset^n \setminus B (0, 1)\) and \(\eta = 1\) on \(B (0, \frac{1}{2})\).
  We define for each \(R \in (0, +\infty)\), the function \(u^R : \Rset^n \to \Rset\), by setting for each \(x \in \Rset^n\)
  \[
  u^R (x) \defeq \eta \Bigl( \frac{x}{R}\Bigr) \biggl( u (x) - \fint_{B (0, R)} u\biggr).
  \]
  By \cref{lemma_cut_domain}, 
  we have for each \(R \in (0, +\infty)\),
  \begin{equation}
    \label{eq_feifieshiiXa9euphu1}
  \iint\limits_{\Rset^n \times \Rset^n}
  \min_{1 \le i \le \ell} \frac{\abs{u^R (x)- u^R (y)}^{p_i}}{\abs{x - y}^{n + s_i p_i}}
  \dif x \dif y
  \le 
  \C 
  \iint\limits_{B (0, R) \times B (0, R)}
  \min_{1 \le i \le \ell} \frac{\abs{u (x)- u (y)}^{p_i}}{\abs{x - y}^{n + s_i p_i}}
  \dif x \dif y.
\end{equation}
Moreover for each \(R \in (0, +\infty)\), since \(u^R = 0\) in \(\Rset^n \setminus B_R\), by \cref{lemma_Holder_min}
\[
\begin{split}
\min_{1 \le i \le \ell}
\Biggl( R^n
\fint_{B (0, R)} \abs{u^R}\Biggr)^{p_i}
&\le \C  \fint_{B (0, R)} \min_{1 \le i \le \ell} R^{n p_i} \abs{u^R}^{p_i}\\
&\le \C  \int_{B (0, 2R) \setminus B (0, R)} \int_{B (0, R)}  \min_{1 \le i \le \ell} \frac{\abs{u^R (x) - u^R (y)}^{p_i}}{R^{2 n-np_{i}}} \dif x \dif y \\
&\le \C \int_{B (0, 2 R) \setminus B (0, R)} \int_{B (0, R)}  \min_{1 \le i \le \ell}
\frac{\abs{u^R (x) - u^R (y)}^{p_i}}{R^{n-np_{i} - s_i p_i}\abs{x - y}^{n + s_i p_i}}\dif x \dif y \\
&<+\infty,
\end{split}
\]
and therefore
\[
\int_{\Rset^n} \frac{\abs{u^R (x)}}{1 + \abs{x}^n} \dif x 
=\int_{B (0, R)} \frac{\abs{u^R (x)}}{1 + \abs{x}^n} \dif x 
\le \int_{B (0,R)} \abs{u^R} < +\infty.
\]
By the first part of the proof, for each \(R \in (0, +\infty)\), there exist measurable functions \(u_1^R, \dotsc, u_\ell^R : \Rset^n \to \Rset\) such that \(u^R = u_1^R + \dotsb + u_\ell^R\) in \(\Rset^n\) and, by \eqref{eq_feifieshiiXa9euphu1},
\begin{equation}
  \label{eq_oosaip2ierieghaev2P}
  \begin{split}
\iint\limits_{\Rset^n \times \Rset^n}
\max_{1 \le i \le \ell}  \frac{\abs{u^R_i (x)-u^R_i (y)}^{p_i}}{\abs{x - y}^{n + s_i p_i}}
\dif x \dif y
&\le 
\C
\iint\limits_{\Rset^n \times \Rset^n}
\min_{1 \le i \le \ell} \frac{\abs{u^R (x)-u^R (y)}^{p_i}}{\abs{x - y}^{n + s_i p_i}}
\dif x \dif y\\
& \le \C  \iint\limits_{\Rset^n \times \Rset^n}
\min_{1 \le i \le \ell} \frac{\abs{u (x)- u (y)}^{p_i}}{\abs{x - y}^{n + s_i p_i}}
\dif x \dif y
.
\end{split}
\end{equation}

If \(R \ge 2\),  we can assume without loss of generality by adding suitable constants to the functions \(u_1, \dotsc, u_\ell\), that for every \(i \in \{1, \dotsc,\ell\}\),
\[
\int_{B (0, 1)} u^R_i = \kappa\defeq \frac{1}{\ell} \int_{B (0, 1)} u.
\]
This implies in turn that for every \(\rho > 0\),
\begin{equation}
  \label{eq_eacho7Aexoo4eGoDuh2}
\begin{split}
\int_{B (0, \rho)} \abs{u^R_i - \kappa}^{p_i}
&\le \C \iint\limits_{B (0, \rho)  \times B (0, 1)} \abs{u^R_i(x) - u^R_i (y)}^{p_i} \dif y \dif x\\
&\le \C \rho^{n + s_ip_i} \iint\limits_{B (0, \rho)  \times B (0, 1)} \frac{\abs{u^R_i(x) - u^R_i (y)}^{p_i}}{\abs{x - y}^{n +s_i p_i}} \dif y \dif x
\end{split}
\end{equation}
In view of \eqref{eq_oosaip2ierieghaev2P} and \eqref{eq_eacho7Aexoo4eGoDuh2}, for every \(\rho > 0\) and \(i \in \{1, \dotsc, \ell\}\), the family 
\((u^R_i)_{R > 2}\) is bounded in the space \(W^{s_i, p_i} (B (0, \rho))\).
By the Rellich compactness theorem in fractional Sobolev spaces \cite{DiNezza_Palatucci_Valdinoci2012}*{theorem 7.1}, 
there exist functions \(u_1,\dotsc, u_\ell : \Rset^n \to \Rset\) and a sequence \((R_m)_{m \in \Nset}\) diverging to \(+ \infty\) such that the sequence
\((u^{R_m}_i)_{m \in \Nset}\) converges almost everywhere to \(u_i\) in \(\Rset^n\).
This implies in particular that for almost every \(x \in \Rset^n\),
\[
u_1 (x) + \dotsb + u_\ell (x) = \lim_{m \to \infty} u_1^{R_m} (x) + \dotsb + u_\ell^{R_m} (x)
= \lim_{m \to \infty} u^{R_m} (x) = u (x).
\]
Finally, by Fatou's lemma and by \eqref{eq_oosaip2ierieghaev2P}, we have 
\begin{equation*}
  \begin{split}
  \iint\limits_{\Rset^n \times \Rset^n}
  \max_{1 \le i \le \ell}  \frac{\abs{u_i (x)-u_i (y)}^{p_i}}{\abs{x - y}^{n + s_i p_i}}
  \dif x \dif y
 &\le 
  \liminf_{m \to \infty}
  \iint\limits_{\Rset^n \times \Rset^n}
  \max_{1 \le i \le \ell}  \frac{\abs{u^{R_m}_i (x)-u^{R_m}_i (y)}^{p_i}}{\abs{x - y}^{n + s_i p_i}}
  \dif x \dif y\\
  &\le \C   \iint\limits_{\Rset^n \times \Rset^n}
  \min_{1 \le i \le \ell} \frac{\abs{u (x)- u (y)}^{p_i}}{\abs{x - y}^{n + s_i p_i}}
  \dif x \dif y.\qedhere
\end{split}
\end{equation*}

\end{proof}

\section{Decomposition of functions in bounded domains and on manifolds}

The aim of this section is to prove the counterpart of \cref{proposition_decomposition_Rn} in bounded domains and compact manifolds having possibly a boundary.
\begin{proposition}
  \label{proposition_decomposition_bounded_domains}
  Let \(n \in \Nset^*\), let \(\ell \in \Nset^*\), let \(s_1, \dotsc, s_\ell \in (0, 1)\) and let \(p_1, \dotsc, p_\ell \in [1, + \infty)\). Let \(\Omega\) be a bounded domain with a smooth boundary or a smooth compact manifolds with (a possibly empty) boundary.
  There exists a constant \(C\) such that for every
  measurable function \(u : \Omega \to \Rset\), 
  there exist measurable functions \(u_1, \dotsc, u_\ell : \Omega \to \Rset\) such that \(u = u_1 + \dotsb + u_\ell\) on \(\Omega\) and 
  \[
  \iint\limits_{\Omega \times \Omega}
  \max_{1 \le i \le \ell}  \frac{\abs{u_i (x)-u_i (y)}^{p_i}}{d (x, y)^{n + s_i p_i}}
  \dif x \dif y
  \le 
  C
  \iint\limits_{\Omega \times \Omega}
  \min_{1 \le i \le \ell} \frac{\abs{u (x)-u (y)}^{p_i}}{d (x, y)^{n + s_i p_i}}
  \dif x \dif y
  .
  \]
\end{proposition}

The constant in the previous proposition depends on the domain or the manifold \(\Omega\) and also on the number \(\ell\) and the parameters \(s_1, \dotsc, s_\ell\) and \(p_1, \dotsc, p_\ell\) in the same way as in proposition \ref{proposition_decomposition_Rn}.

We first remark that the boundedness of the integral in the right hand-side of \cref{proposition_decomposition_bounded_domains} implies integrability and thus it will make sense to prescribe average values.

\begin{lemma} 
  \label{lemma_integrability}
Let \(n \in \Nset^*\), \(\ell \in \Nset^*\), \(s_1, \dotsc, s_\ell \in (0, 1)\) and \(p_1, \dotsc, p_\ell \in [1, + \infty)\). 
If \(\Omega\) is a bounded manifold, then there exists a constant \(C\) such that 
\[
\iint\limits_{\Omega \times \Omega}
\abs{u (x) - u (y)} \dif x \dif y 
\le 
C \Biggl( 1 + 
\iint\limits_{\Omega \times \Omega}
\min_{1 \le i \le \ell}  \frac{\abs{u (x)-u (y)}^{p_i}}{d (x, y)^{n + s_i p_i}}
\dif x \dif y\Biggr).
\]
\end{lemma}
\begin{proof}
  \resetconstant
  We observe that since \(\Omega\) is bounded and \(p_i \ge 1\), we have for every \(i \in \{1, \dotsc, \ell\}\),
\[
\abs{u(x) - u (y)}
\le \C 
\Bigl(1 + \frac{\abs{u (x)-u (y)}^{p_i}}{d (x, y)^{n + s_i p_i}}\Bigr),
\]
and hence the inequality follows by taking the minimum over \(i \in \{1, \dotsc, \ell\}\) and by integrating over \(\Omega \times \Omega\).
\end{proof}

\begin{lemma}
  \label{lemma_extension_ball}
Let \(n \in \Nset^*\), let \(\ell \in \Nset^*\), let \(s_1, \dotsc, s_\ell \in (0, 1)\) and let \(p_1, \dotsc, p_\ell \in [1, + \infty)\). 
There exists a constant \(C\) such that for every measurable function \(u : B (0, R) \to \Rset\) is measurable, there exists a measurable function 
\(\Tilde{u} : B (0, 2 R) \to \Rset\), such that \(\Tilde{u} = u\) in \(B (0, R)\) and 
\[
  \iint\limits_{B (0, 2R) \times B (0, 2R) }
  \min_{1 \le i \le \ell}  \frac{\abs{\Tilde{u} (x)-\Tilde{u} (y)}^{p_i}}{\abs{x - y}^{n + s_i p_i}}
  \dif x \dif y
  \le 
  C 
  \iint\limits_{B (0, R) \times B (0, R)}
  \min_{1 \le i \le \ell}  \frac{\abs{u (x)- u (y)}^{p_i}}{\abs{x - y}^{n + s_i p_i}}
  \dif x \dif y.
\]
\end{lemma}
\begin{proof}
  We define the function \(\Tilde{u} : B (0, 2R) \to \Rset\) for each \(x \in B (0, 2R)\) by 
\[
\Tilde{u} (x)
\defeq 
\left\{
\begin{aligned}
  &u (x) &&\text{if \(x \in B (0, R)\)},\\
  &u \bigl(\tfrac{R^2}{\abs{x}^2} x \bigr) &&\text{if \(x \in B (0, 2R) \setminus B (0, R)\)}.
\end{aligned}
\right.
\]
We compute
\[
\begin{split}
\iint\limits_{B (0, 2R) \times B (0, 2R) }&
\min_{1 \le i \le \ell}  \frac{\abs{\Tilde{u} (x)-\Tilde{u} (y)}^{p_i}}{\abs{x - y}^{n + s_i p_i}}
\dif x \dif y\\
&= \iint\limits_{B (0, R) \times B (0, R) }
\min_{1 \le i \le \ell}  \frac{\abs{\Tilde{u} (x)-\Tilde{u} (y)}^{p_i}}{\abs{x - y}^{n + s_i p_i}}
\dif x \dif y\\
&\qquad + 2\iint\limits_{B (0, R) \times (B (0, 2R) \setminus B (0, R))}
\min_{1 \le i \le \ell}  \frac{\abs{\Tilde{u} (x)-\Tilde{u} (y)}^{p_i}}{\abs{x - y}^{n + s_i p_i}}
\dif x \dif y\\
& \qquad + \iint\limits_{(B (0, 2R) \setminus B (0, R)) \times (B (0, 2R) \setminus B (0, R))}
\min_{1 \le i \le \ell}  \frac{\abs{\Tilde{u} (x)-\Tilde{u} (y)}^{p_i}}{\abs{x - y}^{n + s_i p_i}}
\dif x \dif y,
\end{split}
\]
and the conclusion follows from suitable changes of variables from \(B (0, 2R) \setminus B (0, R)\) to \(B (0, R) \setminus B(0, R/2)\).
\end{proof}

\begin{proof}[Proof of \cref{proposition_decomposition_bounded_domains} when \(\Omega = B (0, R)\)]
  \resetconstant
Let \(u : B (0, R) \to \Rset\) be a measurable function and 
let \(\Tilde{u} : B (0, 2R) \to \Rset\) be the extension given by \cref{lemma_extension_ball}.
We define the function
\(v \defeq \eta (\Tilde{u} - \fint_{B (0, 2R)} \Tilde{u}) : \Rset^n \to \Rset\), 
where the function \(\eta \in C^1_c (\Rset^n)\) satisfies \(\eta = 1\) in \(B (0, R)\) and \(\eta = 0\) on \(\Rset^n \setminus B (0, 2R)\).
By \cref{lemma_cut_domain}, we have 
\[
 \iint\limits_{\Rset^n \times \Rset^n}
 \min_{1 \le i \le \ell}  \frac{\abs{v (x)- v (y)}^{p_i}}{\abs{x - y}^{n + s_i p_i}}
 \dif x \dif y
 \le 
 \C 
 \iint\limits_{B (0, R) \times B (0, R) }
 \min_{1 \le i \le \ell}  \frac{\abs{u (x)- u (y)}^{p_i}}{\abs{x - y}^{n + s_i p_i}}
 \dif x \dif y,
\]
and \(v = u - \fint_{B (0, 2R)} \Tilde{u}\) in \(B (0, R)\).
Let \(v_1, \dotsc, v_\ell : \Rset^n \to \Rset\) be measurable functions given by \cref{proposition_decomposition_Rn} such that \(v_1 + \dotsb + v_\ell = v\) and 
\[
\iint\limits_{\Rset^n \times \Rset^n}
\max_{1 \le i \le \ell}  \frac{\abs{v_i (x)-v_i (y)}^{p_i}}{\abs{x - y}^{n + s_i p_i}}
\dif x \dif y
\le 
C
\iint\limits_{\Rset^n \times \Rset^n}
\min_{1 \le i \le \ell} \frac{\abs{v (x)- v (y)}^{p_i}}{\abs{x - y}^{n + s_i p_i}}
\dif x \dif y
.
\]
We conclude by setting 
\(u_i \defeq (v_i + \frac{1}{\ell} \fint_{B (0, 2 R)} \Tilde{u})\vert_{B (0, R)}\).
\end{proof}

\begin{proof}[Proof of \cref{proposition_decomposition_bounded_domains} in the general case]
  \resetconstant
Since \(\Omega\) is a compact Lipschitz manifold, there exist 
\(N \in \Nset\), and for \(k \in \{1, \dotsc, N\}\), a bi-Lipschitz homeomorphism \(\psi_k : U_k \to \Rset^m\) such that  either \(\psi_k (U_k) = B (0, 1) \subset \Rset^n\) or \(\psi_k (U_k) = B (0, 1) \cap \Rset^{n - 1} \times [0, +\infty)\)
and such that \(\Omega=\cup_{k=1}^N U_k\) . 
We take a partition of unity \((\varphi_k)_{1 \le k \le N}\) associated to the sets \(U_k\), that is, for every \(k \in \{1, \dotsc, N\}\), \(\varphi_k \in C^1 (\Omega)\) and \(\varphi_k = 0\) in \(\Omega \setminus U_k\),
and \(\sum_{i = 1}^{N} \varphi_k = 1\).
For each \(k \in \{1, \dotsc, N\}\) we define the function \(v^k\defeq u \compose \psi_k^{-1} : \psi_k (U_k) \to \Rset\). 
By the change of variable formula on a Riemannian manifold, we have for each \(k \in \{1, \dotsc, N\}\), 
\begin{multline*}
\iint\limits_{\psi_k (U_k)\times \psi_k (U_k)} 
  \min_{1\leq i\leq \ell}\frac{|v^k(x)-v^k(y)|^{p_i}}{\abs{x - y}^{n+s_ip_i}} 
  \dif x \dif y \\
    =
  \iint\limits_{U_k \times U_k}  \min_{1\leq i\leq \ell}
  \frac{|u(x)-u(y)|^{p_i}}{| \psi_k(x)-\psi_k(y)|^{n+s_ip_i}}
  \,
  J \psi_k (x) \,J \psi_k (y)\, \dif x \dif y
  ,
\end{multline*}
where the Jacobian is defined for \(x \in \Omega\) as \(J \psi_k (x) \defeq \det ([D \psi_k(x)]^* \compose D \psi_k(x))\),
with the adjoint \([D \psi_k(x)]^*\) being computed with respect to the Euclidean metric and the Riemannian metric on \(\Omega\).
Since \(\psi_k\) is bi-Lipschitz the Jacobian \(J \psi_k \) is bounded and \(d_\Omega (x,y) \leq \C \abs{ \psi_k(x)-\psi_k(y)}\). Thus
\begin{equation*}
  \begin{split}
\iint\limits _{\psi_k (U_k)\times \psi_k (U_k)} \min_{1\leq i\leq \ell}\frac{|v^k(x)-v^k(y)|^{p_i}}{\abs{x - y}^{n+s_ip_i}} \dif x \dif y 
& \leq  \C \iint\limits_{U_k \times U_k} \min_{1\leq i\leq \ell} \frac{|u(x)-u(y)|^{p_i}}{d_\Omega(x,y)^{n+s_ip_i}}\dif x \dif y \\
& \leq \C \iint\limits_{\Omega \times \Omega} \min_{1\leq i\leq \ell} \frac{|u(x)-u(y)|^{p_i}}{d_\Omega(x,y)^{n+s_ip_i}} \dif x \dif y.
\end{split}
\end{equation*}

Since the proposition is proved on a ball and the set \(\psi_k (U_k)\) is either a ball or a half-ball which is 
the image of a ball under a bi-Lipschitz homeomorphisms, for every \(k \in \{1, \dotsc, N\}\), 
there exist measurable functions \(v^k_1,\dotsc,v^k_\ell : \psi_k (U_k) \to \Rset\) such that 
\(v^k = v^k_1 + \dotsb + v^k_\ell\) on \(\psi_k (U_k)\) and 
\[
\iint\limits_{\psi_k (U_k)\times \psi_k (U_k)}\!\! \max_{1\leq i\leq \ell}\frac{|v^k_i(x)-v^k_i(y)|^{p_i}}{\abs{x - y}^{n+s_ip_i}} \dif x \dif y 
\le \C \!\!
\iint\limits_{\psi_k (U_k)\times \psi_k (U_k)} \!\! \min_{1\leq i\leq \ell}\frac{|v^k(x)-v^k(y)|^{p_i}}{\abs{x - y}^{n+s_ip_i}} \dif x \dif y. 
\]
Moreover, we can assume that for each \(k \in \{1, \dotsc, N\}\) and \(i \in \{1, \dotsc, \ell\}\) we have
\begin{equation}
  \fint_{U_k} v^k_i\compose \psi_k = \frac{1}{\ell} \fint_{U_k} u.
\end{equation}
We define for each \(i \in \{1, \dotsc, \ell\}\) the function
\[
\Tilde{u}_i \defeq \sum_{k = 1}^N \varphi_k \Bigl(v^k_i \compose \psi_k - \fint_{U_k} v^k_i\compose \psi_k \Bigr).
\]
Since the map \(\psi_k\) is bi-Lipschitz and by \cref{lemma_cut_domain}, we have 
\[
\begin{split}
\iint\limits_{\Omega \times \Omega} \frac{|\Tilde{u}_i(x)-\Tilde{u}_i(y)|^{p_i}}{d_\Omega( x, y)^{n+s_i p_i}} \dif x \dif y 
&\le \C \sum_{k = 1}^N \iint\limits_{U_k \times U_k} \frac{|\Tilde{u}_i(x)-\Tilde{u}_i(y)|^{p_i}}{d_\Omega (x, y)^{n+s_i p_i}} \dif x \dif y \\
&\le \C \sum_{k = 1}^N \; \iint\limits_{\psi_k (U_k)\times \psi_k (U_k)} \frac{|v^k_i(x)-v^k_i(y)|^{p_i}}{\abs{x - y}^{n+s_ip_i}} \dif x \dif y.
\end{split}
\]
If we define the low frequency component
\[
  \Tilde{u}_0 \defeq \sum_{k = 1}^N \varphi_k \biggl(\fint_{U_k} u - \fint_{\Omega} u\biggr),
\]
we have on \(\Omega\)
\[
u = \fint_{\Omega} u + \Tilde{u}_0 + \sum_{i = 1}^\ell \Tilde{u}_i.
\]

We compute now for each \(k \in \{1, \dotsc, N\}\),
\[
\biggabs{\fint_{U_k} u - \fint_{\Omega} u}
\le 
\C \iint\limits_{\Omega \times \Omega} \abs{u (x) - u (y)},
\]
and thus by \cref{lemma_Holder_min}, 
\[
\begin{split}
\min_{1 \le i \le \ell}
\norm{\nabla \Tilde{u}_0}_{L^{+\infty}}^{p_i}
&\le \C \iint\limits_{\Omega \times \Omega} \min_{1 \le i \le \ell} \abs{u (x) - u (y)}^{p_i} \dif x \dif y\\
& \le \C \iint\limits_{\Omega \times \Omega} \min_{1 \le i \le \ell} \frac{\abs{u (x) - u (y)}^{p_i}}{d_\Omega (x, y)^{n + s_ip_i}} \dif x \dif y.
\end{split}
\]
Since for every \(i \in \{1, \dotsc, \ell\}\),
\[
\iint\limits_{\Omega \times \Omega} \frac{\abs{\Tilde{u}_0 (x) - \Tilde{u}_0 (y)}^{p_i}}{d_\Omega (x, y)^{n + s_ip_i}} \dif x \dif y
\le 
\C
\iint \limits_{\Omega \times \Omega} \frac{ \norm{\nabla \Tilde{u}_0}_{L^{+\infty}}^{p_i}}{d_\Omega (x, y)^{n  - (1 -s_i)p_i}} \dif x \dif y
\le 
\C \norm{\nabla \Tilde{u}_0}_{L^{+\infty}}^{p_i},
\]
it follows then that 
\[
\begin{split}
  \min_{1 \le i \le \ell}
\iint\limits_{\Omega \times \Omega} \frac{\abs{\Tilde{u}_0 (x) - \Tilde{u}_0 (y)}^{p_i}}{d_\Omega (x, y)^{n + s_ip_i}} \dif x \dif y
&\le \C \iint\limits_{\Omega \times \Omega} \min_{1 \le i \le \ell} \abs{u (x) - u (y)}^{p_i} \dif x \dif y\\
&\le \C \iint\limits_{\Omega \times \Omega} \min_{1 \le i \le \ell} \frac{\abs{u (x) - u (y)}^{p_i}}{d_\Omega (x, y)^{n + s_ip_i}} \dif x \dif y.
\end{split}
\]
The conclusion then follows.
\end{proof}


\begin{bibdiv}
  \begin{biblist}
    
    \bib{Bethuel_Chiron2007}{incollection}{
      author={Bethuel, Fabrice},
      author={Chiron, David},
      title={Some questions related to the lifting problem in {S}obolev
        spaces},
      date={2007},
      booktitle={Perspectives in nonlinear partial differential equations},
      series={Contemp. Math.},
      volume={446},
      publisher={Amer. Math. Soc., Providence, RI},
      pages={125\ndash 152},
      doi={10.1090/conm/446/08628},
%       review={\MR{2373727}},
    }
    
    \bib{Bourgain_Brezis2003}{article}{
      author={Bourgain, Jean},
      author={Brezis, Ha\"{i}m},
      title={On the equation {${\rm div}\, Y=f$} and application to control of
        phases},
      date={2003},
      ISSN={0894-0347},
      journal={J. Amer. Math. Soc.},
      volume={16},
      number={2},
      pages={393\ndash 426},
      doi={10.1090/S0894-0347-02-00411-3},
%       review={\MR{1949165}},
    }
    
    
    
    \bib{Bourgain_Brezis_Mironescu2000}{article}{
   author={Bourgain, Jean},
   author={Brezis, Haim},
   author={Mironescu, Petru},
   title={Lifting in Sobolev spaces},
   journal={J. Anal. Math.},
   volume={80},
   date={2000},
   pages={37--86},
   issn={0021-7670},
%    review={\MR{1771523}},
   doi={10.1007/BF02791533},
}
    \bib{Coifman_Fefferman_1974}{article}{
      author={Coifman, R. R.},
      author={Fefferman, C.},
      title={Weighted norm inequalities for maximal functions and singular
        integrals},
      journal={Studia Math.},
      volume={51},
      date={1974},
      pages={241--250},
      issn={0039-3223},
%       review={\MR{0358205}},
      doi={10.4064/sm-51-3-241-250},
    }
    \bib{DiNezza_Palatucci_Valdinoci2012}{article}{
      author={Di~Nezza, Eleonora},
      author={Palatucci, Giampiero},
      author={Valdinoci, Enrico},
      title={Hitchhiker's guide to the fractional {S}obolev spaces},
      date={2012},
      ISSN={0007-4497},
      journal={Bull. Sci. Math.},
      volume={136},
      number={5},
      pages={521\ndash 573},
      doi={10.1016/j.bulsci.2011.12.004},
%       review={\MR{2944369}},
    }
    
    \bib{Gagliardo_1957}{article}{
      author={Gagliardo, Emilio},
      title={Caratterizzazioni delle tracce sulla frontiera relative ad alcune
        classi di funzioni in \(n\) variabili},
      %    language={Italian},
      journal={Rend. Sem. Mat. Univ. Padova},
      volume={27},
      date={1957},
      pages={284--305},
      issn={0041-8994},
      %    review={\MR{0102739}},
    }    
    
    \bib{Hardy_Littlewood_Polya_1952}{book}{
      author={Hardy, G. H.},
      author={Littlewood, J. E.},
      author={P\'{o}lya, G.},
      title={Inequalities},
      edition={2},
      publisher={Cambridge University Press},
      date={1952},
      pages={xii+324},
%       review={\MR{0046395}},
    }
    
   
    \bib{Mironescu2008}{article}{
      author={Mironescu, Petru},
      title={Lifting default for {$\mathbb{S}^1$}-valued maps},
      date={2008},
      ISSN={1631-073X},
      journal={C. R. Math. Acad. Sci. Paris},
      volume={346},
      number={19--20},
      pages={1039\ndash 1044},
      doi={10.1016/j.crma.2008.08.001},
%       review={\MR{2462045}},
    }

    \bib{Mironescu_preprint}{article}{
      author={Mironescu, Petru},
      title={Lifting of \(\mathbb{S}^1\)-valued maps in sums of Sobolev spaces},
      eprint={https://hal.archives-ouvertes.fr/hal-00747663},
    }
    
    \bib{Mironescu_Russ2015}{article}{
      author={Mironescu, Petru},
      author={Russ, Emmanuel},
      title={Traces of weighted Sobolev spaces. Old and new},
      date={2015},
      ISSN={0362-546X},
      journal={Nonlinear Anal.},
      volume={119},
      pages={354\ndash 381},
      doi={10.1016/j.na.2014.10.027},
%       review={\MR{3334194}},
    }
    
    \bib{Nguyen2008}{article}{
      author={Nguyen, Hoai-Minh},
      title={Inequalities related to liftings and applications},
      date={2008},
      ISSN={1631-073X},
      journal={C. R. Math. Acad. Sci. Paris},
      volume={346},
      number={17-18},
      pages={957\ndash 962},
      doi={10.1016/j.crma.2008.07.026},
%       review={\MR{2449635}},
    }
    
    \bib{Stein1993}{book}{
      author={Stein, Elias~M.},
      title={Harmonic analysis: real-variable methods, orthogonality, and
        oscillatory integrals},
      series={Princeton Mathematical Series},
      publisher={Princeton University Press}, 
      address={Princeton, N.J.},
      date={1993},
      volume={43},
      ISBN={0-691-03216-5},
      contribution={with the assistance of T. S. Murphy},
%       series={Monographs in Harmonic
%         Analysis}, 
%       volume={III},
%       review={\MR{1232192}},
    }
    
    \bib{Uspenskii}{article}{
      author={Uspenski\u{\i}, S. V.},
      title={Imbedding theorems for classes with weights},
      language={Russian},
      journal={Trudy Mat. Inst. Steklov.},
      volume={60},
      date={1961},
      pages={282--303},
      issn={0371-9685},
      translation={
        journal={Am. Math. Soc. Transl.},
        volume={87},
        pages={121--145},
        date={1970},
      },
  }
    
    \bib{Zhou2015}{article}{
   author={Zhou, Yuan},
   title={Fractional Sobolev extension and imbedding},
   journal={Trans. Amer. Math. Soc.},
   volume={367},
   date={2015},
   number={2},
   pages={959--979},
   issn={0002-9947},
%    review={\MR{3280034}},
   doi={10.1090/S0002-9947-2014-06088-1},
}
    
  \end{biblist}
\end{bibdiv}
\end{document}